\documentclass[ 10pt, a4paper, reqno]{amsart}
\usepackage[UKenglish]{babel} 
\usepackage[T1]{fontenc}
\usepackage[utf8]{inputenc} 
\usepackage{amsmath, amssymb, wasysym, amsthm}
\usepackage{enumerate}
\usepackage{float}
\usepackage{cite}
\usepackage{mathrsfs}
\usepackage{mathtools} 
\usepackage[pdfencoding=unicode]{hyperref} 

\newtheorem{theorem}{Theorem}[section]
\newtheorem*{theorem*}{Theorem}
\newtheorem{theorema}{Theorem}

\newtheorem{proposition}[theorem]{Proposition}

\newtheorem{lemma}[theorem]{Lemma}
\newtheorem{corollary}[theorem]{Corollary}

\newtheorem*{conjecture}{Conjecture}
\theoremstyle{definition}
\newtheorem{definition}[theorem]{Definition}
\theoremstyle{remark}

\newtheorem{remark}[theorem]{Remark}

\DeclareMathOperator{\interior}{int}

\DeclareMathOperator{\Scal}{Scal}

\numberwithin{equation}{section}

\newcommand{\R}{\mathbb{R}}

\newcommand{\del}{\partial}
\renewcommand{\d}{\, \textup{d}}

\newcommand{\II}{\textup{I\hspace{-.7pt}I}}

\newlength{\vierz}
\date{\today}


\begin{document}

\title[Scalar curvature rigidity]{Scalar curvature rigidity for locally conformally flat manifolds with boundary}
\author{Fabian M. Spiegel}
\address{Max Planck Institute for Mathematics, Vivatsgasse 7,
53111 Bonn, Germany}
\email{spiegel@mpim-bonn.mpg.de}

\begin{abstract}
Inspired by the work of F. Hang and X. Wang and partial results by S. Raulot, we prove a scalar curvature rigitidy result for locally conformally flat manifolds with boundary in the spirit of the well-known Min-Oo conjecture. 
\end{abstract}

\maketitle

\section{Introduction}

The well-known Min-Oo conjecture has fascinated mathematicians for over a decade. Motivated by the positive mass theorem (first proven by Schoen and Yau in \cite{SY:PMT} and \cite{SY:PMT2}) and similar results for nonpositive curvature (e.g. by Bartnik \cite{Bartnik} and himself \cite{Min-Oo2}), Min-Oo conjectured the following \cite[Theorem 4]{Min-Oo}:

\begin{conjecture}
Let $(M,g)$ be an $n$-dimensional compact connected Riemannian manifold with boundary. Assume that 
\begin{enumerate}[(i)]
\item $\Scal(g) \geq \Scal(g_{S^n}) = n(n-1)$,
\item The boundary $\del M$ is totally geodesic and isometric to $S^{n-1}$.
\end{enumerate}
Then $(M, g)$ is isometric to the upper hemisphere $S^n_+ \coloneqq \{x \in S^n \mid x_{n+1} \geq 0 \}$.
\end{conjecture}

Min-Oo's conjecture was widely believed to be true but proven wrong in 2011 when Brendle, Marques and Neves \cite{BMN} were able to construct a counterexample valid in dimensions $n \geq 3$. Min-Oo's conjecture is true in dimension two, by an old result due to Topogonov \cite{Topogonov}. However, several partial results have been obtained and modified versions of Min-Oo's conjecture hold in many special cases, see e.g. \cite{HangWang}, \cite{HangWang2}, \cite{Eichmair}, \cite{HuangWu}, \cite{BrendleMarques}, \cite{MiaoTam} and the survey article \cite{Brendle}. For example, in \cite{HangWang}, Hang and Wang showed that Min-Oo's conjecture is true for metrics in the conformal class of the standard metric:

\begin{theorem} \label{Thm:HangWang}
Let $g = e^{2\phi}g_{S^n}$ be a $C^2$ metric on $S^n_+$. Assume that
\begin{enumerate}[(i)]
\item $\Scal (g) \geq n(n-1)$ everywhere,
\item The boundary is isometric to $S^{n-1}$.
\end{enumerate}
Then $g$ is isometric to $g_{S^n}$.
\end{theorem}

In \cite{HangWang2}, they were able to prove a similar result for domains in $S^n_+$. The proofs rely on the analysis of the equations for conformal scalar and mean curvature. If $n \geq 3$ and $g$,  $\tilde{g} = u^{\frac{4}{n-2}}g$ are conformally equivalent metrics, then
\begin{align*}
\frac{n-2}{4(n-1)}\Scal(\tilde{g}) \; u^{\frac{n+2}{n-2}} &= \frac{n-2}{4(n-1)} \Scal(g) u - \Delta u, \\
\frac{n-2}{2} H(\tilde{g}) u^{\frac{n}{n-2}}  &= \frac{n-2}{2} H(g) u + \frac{\del u}{\del \eta},
\end{align*}
where $H$ denotes the mean curvature computed with respect to the inner unit normal $\nu = -\eta$.

Motivated by Hang and Wang's results, Raulot \cite{Raulot} was able to extend Theorem \ref{Thm:HangWang} to a class of locally conformally flat manifolds, that is, manifolds which are not globally conformally equivalent to the upper hemisphere but locally look like a conformal deformation of the sphere (for a precise definition see Section \ref{Sec:Developing map}). Using the Chern-Gauß-Bonnet formula, he proved:

\begin{proposition}[Corollaire 1 in \cite{Raulot}]
Let $(M^n, g)$ be a compact Riemannian manifold with boundary of dimension $n= 4$ or $n=6$ with $\chi(M) =1$. Suppose that the boundary $\del M$ is umbilic with non-negative mean curvature and isometric to the round sphere $S^{n-1}$. If $(M, g)$ is locally conformally flat with scalar curvature $\Scal \geq n(n-1)$, then $(M, g)$ is isometric to the standard hemisphere. 
\end{proposition}

In this paper, we prove a scalar curvature rigidity theorem for locally conformally flat manifolds in the spirit of Min-Oo's conjecture without assumptions on the dimension or the Euler characteristic.

As Min-Oo's conjecture is true in dimension two, we always assume $n \geq 3$.\\

To state our main result, we fix the following definitions: Let $p \in S^n$, $0 < \rho \leq \frac{\pi}{2}$ and  $D_\rho \coloneqq D_\rho(p) \coloneqq \{x \in S^n \mid d^{S^n}(x, p) < \rho \}$ be the geodesic ball of radius $\rho$ around $p$ in $S^n$. Let $H_\rho = \cot(\rho)$ be the mean curvature of the boundary $\Sigma_\rho \coloneqq \del D_\rho$. Note that $\Sigma_\rho$ is isometric to a sphere of radius $\sin(\rho)$.

Then we have 

\begin{theorema} \label{Thm:Theorema}
Let $(M^n, g)$, $n \geq 3$, be a compact connected locally conformally flat Riemannian manifold with boundary. Assume that
\begin{enumerate}[i)]
\item $\Scal (g) \geq n(n-1)$ everywhere, 
\item The boundary $\del M$ is umbilic with mean curvature $H(g) \geq H_\rho$ and isometric to $\Sigma_\rho$.
\end{enumerate}
Then $(M, g)$ is isometric to $\overline{D_\rho}$ with the standard metric.
\end{theorema}

For the reader's convenience, we collect all necessary definitions and conventions in Section \ref{Sec:Developing map}.

\begin{remark}
 The bound $\rho \leq \frac{\pi}{2}$ is optimal, see \cite[Claim 2.5]{HangWang}. Furthermore, using \cite[Claim 3.5]{HangWang}, one can show that the assumption on the mean curvature can be dropped provided $M$ is simply-connected and $\rho = \frac{\pi}{2}$, see the proof in Section \ref{SSec:Simply-connected}. This is impossible if $\rho \neq \frac{\pi}{2}$ as we cannot distinguish $\overline{D_{\rho}}$ from $\overline{D_{\pi - \rho}}$ or from geodesic balls in smaller spheres. 
\end{remark}

\begin{remark} \label{Rem:Obata}
The assumption on the boundary being isometric to $\Sigma_\rho$ can be relaxed to $\del M$ being simply-connected and of constant scalar curvature $\csc(\rho)^2(n-1)(n-2)$, see Corollary \ref{Cor:Obata}.
\end{remark}

This paper is structured as follows: In Section \ref{Sec:Developing map}, we present all necessary definitions and basics on locally conformally flat manifolds and conformal transformations. We prove Theorem \ref{Thm:Theorema} under the additional assumption on $M$ being simply-connected as we find the argumentation instructive and motivating. In Section \ref{Sec:Injectivity}, we show that the universal covering of any manifold satisfying the assumptions of Theorem \ref{Thm:Theorema} can be conformally embedded in $S^n$ using a deep result by Schoen and Yau. In Section \ref{Sec:Extension}, we conclude Theorem \ref{Thm:Theorema} using similar methods as Hang and Wang in \cite{HangWang} and \cite{HangWang2}.

\subsection*{Acknowledgment}
It is a pleasure to thank my advisor Werner Ballmann for suggesting this problem to me and many helpful and enlightening discussions. I am also very grateful to Saskia Voß for mathematical conversations and proofreading. Moreover I gratefully acknowledge the support and hospitality of
the Max-Planck-Institute for Mathematics in Bonn.

\section{The developing map} \label{Sec:Developing map}

In this section we fix some definitions and present basics on locally conformally flat manifolds and Möbius transformations. For more background see e.g. \cite{SchoenYau} and \cite{Ratcliffe}. As a corollary, we obtain Theorem \ref{Thm:Theorema} under the additional assumption that $M$ is simply-connected. We start with some basic definitions:\\

\begin{definition}
We say that a hypersurface $\Sigma \subseteq M$ is \emph{umbilic} if the second fundamental form is a multiple of the first fundamental form, that is
\[
\II = H g|_{T\Sigma \times T\Sigma}.
\]
\end{definition}
Note that being umbilic is a conformal invariant; if $\Sigma \subseteq M$ is umbilic with respect to a metric $g$ then it is also umbilic with respect to all metrics of the form $e^{2 \phi}g$.

Throughout this paper, all second fundamental forms and mean curvatures will be computed with respect to the \emph{inner} unit normal $\nu = -\eta$.

\subsection{Locally conformally flat manifolds}

\begin{definition}
A $C^k$-Riemannian metric $g$ on a smooth manifold $M$ is called \emph{locally conformally flat} if for every point $p \in M$, there exists a neighbourhood $U$ of $p$ and $\phi \in C^k (U)$ such that the metric $e^{2\phi}g$ is flat on $U$.
\end{definition}

\begin{remark}
In dimension two, every Riemannian manifold is locally conformally flat, due to the existence of so-called \emph{isothermal} coordinates. A three-dimensional Riemannian manifold is locally conformally flat if and only if its Cotton tensor $C$ vanishes while in dimension $n \geq 4$, a Riemannian manifold is locally conformally flat if and only if the Weyl tensor $W$ vanishes.

Note that the Weyl tensor always vanishes in dimension $n \leq 3$ while in dimensions $n \geq 4$, Weyl and Cotton tensor are related via $C_{ijk}  =\frac{n-2}{n-3} \nabla^a W_{aijk} $, so $W=0$ implies $C = 0$.
\end{remark}

\begin{definition} An immersion $\Phi \colon (M, g) \to (N,h)$ is called \emph{conformal} if it is angle-preserving or, equivalently, if we may write $\Phi^*h = e^{2 \phi} g$ for some function $\phi$.
\end{definition}

If $(M, g)$ is a locally conformally flat manifold, we obtain locally defined isometries (necessarily of class $C^{k+1}$) $f_U \colon (U, e^{2 \phi}g) \to (f(U), g_{\text{eucl.}}) \subseteq \R^n $ which are conformal with respect to the metric $g$. If $M$ is simply-connected, one can use Liouville's theorem on conformal mappings to glue all these together to a conformal map $f \colon M \to \overline{\R^n}$. Composing with a stereographic projection, we obtain a conformal map $\Phi \colon M \to S^n$, called the \emph{developing map} which is unique up to conformal transformations of $S^n$.  

We now present an easy topological lemma which we will use several times:

\begin{lemma} \label{Lem:Easy-topo-lemma}
Let $\Phi \colon M \to  N$ be a local homeo-(diffeo-)morphism with $M$ being compact and $N$ being simply-connected. Then $\Phi$ is bijective, i.e.\ a homeo-(diffeo-) morphism.
\end{lemma}

\begin{proof}
As $\Phi$ is a local homeomorphism, $\Phi$ is open, so $\Phi(M)$ is open. Since $M$ is compact, $\Phi(M)$ is compact, hence closed. Thus, $\Phi$ is surjective and therefore a covering. As $N$ is simply-connected, $\Phi$ is bijective.
\end{proof}

As a corollary, we obtain a well-known result due to Kuiper (see \cite{Kuiper}):

\begin{corollary} \label{Cor:Kuiper}
A closed simply-connected locally conformally flat manifold is conformally equivalent to $S^n$.
\end{corollary}

We now show the claim of Remark \ref{Rem:Obata}: First note that an umbilic hypersurface in a locally conformally flat manifold is again locally conformally flat, then use:

\begin{corollary} \label{Cor:Obata}
Let $(M, g)$ be a closed simply-connected locally conformally flat manifold with constant scalar curvature $r^{-2} n(n-1)$. Then $(M, g)$ is isometric with $S^n_r$.  
\end{corollary}

\begin{proof}
By Corollary \ref{Cor:Kuiper}, such $M$ is conformally equivalent to $S^n$. By a theorem of Obata (see \cite[Theorem 6.1]{Obata}), $g$ has constant sectional curvature. 
\end{proof}

\subsection{Möbius transformations} At this point it is worthwhile having a look at certain conformal transformation groups. We define the \emph{Möbius transformation} groups $M(\overline{\R^n})$ and $M(S^n)$ to be the subgroups of the respective diffeomorphism group generated by reflections in hyperspheres where a hyperplane in $\R^n$ is seen as a hypersphere containing infinity. Note that both groups are isomorphic via conjugation with a stereographic projection.

Liouville's theorem states that every (locally defined) conformal transformation of $\overline{\R^n}$ is actually (the restriction of) a Möbius transformation. So the conformal transformation groups of $\overline{\R^n}$ and $S^n$ are $M(\overline{\R^n})$ and $M(S^n)$, respectively.

We set $M(B^n)$ to be the subgroup of $M(\overline{\R^n})$ containing all Möbius transformations preserving the unit ball $B^n$. It turns out that the homomorphism $M(B^{n+1}) \to M(S^{n})$ induced by restriction is an isomorphsim, that is:

\begin{proposition} \label{Prop:Poincare-Extension}
Let $\varphi \in M(S^n)$. Then there exists a unique $\hat{\varphi}\in M(B^{n+1})$ with $\hat{\varphi}|_{S^n} = \varphi$, called \emph{Poincaré extension} of $\varphi$.
\end{proposition}

\begin{proof}
As $M(S^n)$ is generated by reflections in hyperspheres, it is enough to extend these. Let a reflection $\sigma$ in a hypersphere $S \subseteq S^n$ be given. Let $\tilde{S}$ be the generalized hypersphere ($\tilde{S}$ may be a hyperplane) orthogonal to $S^n$ which intersects $S^n$ in $S$. Then the reflection of $\overline{\R^{n+1}}$ in $\tilde{S}$ extends $\sigma$ and leaves $B^n$ invariant. For more details, see \cite[Section 4.4] {Ratcliffe}
\end{proof}

\subsection{The simply-connected case} \label{SSec:Simply-connected} We are now in the position to prove Theorem \nolinebreak \ref{Thm:Theorema} under the condition that $M$ is simply-connected. Although the proof is similar to the non-simply-connected case, we present it separately to illustrate the technique and to argue why the assumption on the mean curvature can be dropped in this case provided $\rho = \frac{\pi}{2}$.

We proceed in three steps: We will first show that the developing map of $M$ is injective and compose with a stereographic projection and a Möbius transformation to obtain a ``nice'' image in $\R^n$, then use the results of Hang and Wang. Note that we will show the injectivity of the developing map without the simply-connectedness of $M$ using a deep result by Schoen and Yau in Section \ref{Sec:Injectivity}.

\begin{proof}[Proof (Theorem \ref{Thm:Theorema}, simply-connected case)] $\;$ \\ \indent
\emph{Step 1:} As $M$ is simply-connected, there exists a developing map $\Phi \colon M \to  S^n$. Since $\del M$ is umbilic and being umbilic is a conformal invariant, the image of $\del M$ must be umbilic in $S^n$, that is, it is contained in a hypersphere $\Sigma \subseteq S^n$. Applying Lemma \ref{Lem:Easy-topo-lemma} to $\Phi|_{\del M}$, we see that $\Phi|_{\del M} \colon \del M \to \Sigma$ is a diffeomorphism. Composing with a Möbius transformation of $S^n$, if necessary, we may assume that $\Sigma$ is the equator $\del S^n_+ = \{x \in S^n \mid x_{n+1} = 0 \}$.

Consider the double manifold $\hat{M} = M \underset{\del M}{\cup} (-M)$. Here we write $-M$ for the second copy of $M$ in $\hat{M}$ in order to distinguish it from $M$ itself. We extend $\Phi$ to a map $\hat{\Phi} \colon \hat{M} \to S^n$ in the following way: We write $\Phi = (\Phi_1, \ldots ,\Phi_{n+1})$ and set
\[
\hat{\Phi}(x) \coloneqq 
\begin{cases}
\Phi(x) \qquad \qquad &\text{ if } x \in M,\\
(\Phi_1 (x), \ldots, \Phi_n(x) , -\Phi_{n+1}(x)) &\text{ if } x \in -M.
\end{cases}
\]
Then $\hat{\Phi}$ is well-defined and continuous because $\Phi_{n+1}(x) = 0$ for $x \in \del M$. Moreover, it is a local homeomorphism. Lemma \ref{Lem:Easy-topo-lemma} implies that $\hat{\Phi}$ is a homeomorphism and hence $\Phi$ is injective. Furthermore, the image is either $S^n_+$ or $S^n_-$. \\

\emph{Step 2:} Let $\{S, N = -S \}$ be a pair of antipodal points. By composing $\Phi$ with a Möbius transformation of $S^n$, we may assume that the image of $\Phi$ is $\overline{D_{\rho}(S)}$. As $\Sigma_{\rho}(S) = \del D_{\rho}(S)$ equipped with the metric $(\Phi^{-1})^*g$ is isometric to $\Sigma_\rho$, we obtain an isometry $f \colon \Sigma_\rho(S) \to (\Sigma_\rho(S) , ((\Phi^{-1})^*g)|_{\Sigma_\rho(S)})$ which is conformal with respect to the standard metric. Using Poincaré extension (see Proposition \ref{Prop:Poincare-Extension}), we can extend $f$ to a Möbius transformation of $S^n$ preserving $D_{\rho}(S)$, call it $F$. Composing $\Phi$ with $F^{-1}$, we may assume that $((\Phi^{-1})^*g)|_{\Sigma_\rho(S)}$ coincides with the standard metric.\\

\emph{Step 3:} By construction, we have obtained a metric $ (\Phi^{-1})^*g$ on $D_{\rho}(S)$ which is conformal to the standard metric and agrees with it on the boundary. We can now apply the results from Hang and Wang:

 If $\rho = \frac{\pi}{2}$, Claim 3.5 from \cite{HangWang} implies that $ (\Phi^{-1})^*g$ is the standard metric, hence $M$ is isometric to $S^n_+$ (without any assumptions on the mean curvature). In all other cases, $H(g) \geq H_\rho$ and \cite[Proposition 1]{HangWang2} imply that  $(\Phi^{-1})^*g$ is the standard metric.
\end{proof}

\section{Injectivity of the developing map} \label{Sec:Injectivity}

If $M$ is not simply-connected, we do not know whether there is a conformal map to $S^n$. However, we can pass to the universal covering $\tilde{M}$ to obtain a developing map $\Phi \colon \tilde{M} \to S^n$. In this section we establish that, under the assumptions of Theorem \ref{Thm:Theorema}, this developing map is injective (see also \cite[Theorem 1.4]{LiNguyen}):

\begin{proposition} \label{Prop:Existence-of-injective-developing-map} Let $(M, g)$ be a compact connected locally conformally flat manifold with boundary. Assume that $M$ has positive scalar curvature and that $\del M$ is umbilic and simply-connected with non-negative mean curvature. Let $\tilde{M} \to M$ be the universal covering. Then there exists an injective conformal map $\Phi \colon \tilde{M} \to S^n$  which is a conformal diffeomorphism onto its image. The image is of the form
\[ 
C = C(\varepsilon_i, p_i,\Lambda) \coloneqq S^n \setminus \left( \bigcup_{i } D_{\varepsilon_i} (p_i) \cup \Lambda \right),  
\] 
where the $D_{\varepsilon_i} (p_i)$ are geodesic balls in $S^n$ with disjoint closures and $\Lambda$ is the so-called \emph{limit set}, a closed subset of Hausdorff dimension at most $\frac{n-2}{2}$.
\end{proposition}

Before presenting the proof, we shortly comment on the limit set: Let $\Gamma \subseteq M(S^n)$ be a subgroup of the conformal transformation group of $S^n$. Then the \emph{limit set} of $\Gamma$ is defined as
\[
\Lambda (\Gamma) \coloneqq \left\{ x \in S^n \mid \text{there exist } x' \in S^n \text{ and } \gamma_i \in \Gamma \text{ such that } \gamma_i x' \to x \right\}.
\]
Now if $(M, g)$ is a locally conformally flat manifold we obtain a developing map $\Phi \colon \tilde{M} \to S^n$, hence $\pi_1(M)$ (viewed as the group of deck transformations) acts on $S^n$ by Möbius transformations. We obtain a homomorphism $\rho \colon \pi_1(M) \to M(S^n)$ called \emph{holonomy representation} of $\pi_1 (M)$ in $M(S^n)$. If the scalar curvature of $M$ is non-negative, one can show that the developing map is injective, the holonomy representation is one-to-one and $\Lambda(\rho(\pi_1(M))) = \del ( \Phi(\tilde{M})) = S^n \setminus \Phi(\tilde{M})$. For more background see \cite{SchoenYau}.\\

The main idea of the proof is to apply a deep result on the injectivity of developing maps by Schoen and Yau (see \cite{SY} and \cite{SchoenYau}) to the double manifold $\hat{M} = M \underset{\del M}{\cup} (-M)$. Then we analyse the behaviour of the boundary components under the developing map. 

\begin{remark}
One may want to use \cite[Theorem 3.5]{SchoenYau} (or \cite[Theorem 4.5]{SY}) which states that the developing map is injective provided $\Scal \geq 0$. However, Schoen and Yau remark that ``for this application it is necessary to extend the positive energy theorems to the case of complete manifolds; that is, assuming that the manifold has an asymptotically flat end and
other ends which are merely complete. This extension will be carried out in a future work.'' To the author's knowledge, such a generalization of the positive mass theorem is widely believed to be true but no such extension has yet been published. Since we do not want to rely on this extension to be true, note that we can use \cite[Proposition 3.3, (i)]{SY} in dimension $n \geq 4$, while in dimension $3$, any orientable manifold is spin so Witten's version of the positive mass theorem (which can be extended to the complete case) implies that \cite[Theorem 4.5]{SY} holds. See also Appendix A of \cite{CH}, where the same problem occured.
\end{remark}

\begin{proof} We proceed in three steps: First we show how to obtain a smooth metric on the double manifold conformal to the canonical one, then apply Schoen and Yau's results to its universal cover. In the last step, we show that the image has the claimed form.\\

\emph{Step 1:} 
If $M$ is conformally diffeomorphic to $S^n_+$, $M$ is simply-connected and the results from Section \ref{SSec:Simply-connected} apply. In all other cases, our assumptions on scalar- and mean curvature imply that the Yamabe invariant of $M$ is positive and hence by Escobar's solution to the Yamabe problem on manifolds with boundary for locally conformally flat manifolds (see \cite{Escobar}), there exists a metric $g'$ in the conformal class of $g$ with (constant) positive scalar curvature and totally geodesic boundary. 

Consider the double manifold $\hat{M} = M \underset{\del M}{\cup} (-M)$. Since $\del M$ is totally geodesic with respect to $g'$, the canonical metric $\hat{g}'$ extending $g'$ is $C^{2, 1}$ as can be seen in Fermi coordinates (see e.g. \cite[Appendix]{Escobar}). One can now check that the results from \cite{SY} (see also \cite{SchoenYau}) hold for $C^{2, 1}$-metrics, but in this special case we can even find a smooth metric with positive scalar curvature conformal to $\hat{g}$:

In fact, for $\varepsilon >0 $ small enough, we take a neighbourhood $U$ of $\del M$ in $M$ diffeomorphic to $\del M \times [0, \varepsilon)$. As the latter is simply-connected, we obtain a conformal map $\Phi_U \colon U \to S^n$. As $\del M$ is umbilic, $\Phi (\del M)$ is also umbilic and hence contained in a hypersphere. By Lemma \ref{Lem:Easy-topo-lemma}, $\Phi_U|_{\del M}$ is a diffeomorphism onto that hypersphere. Hence, as $\Phi_U$ is an immersion, we may shrink $U$ to obtain an injective conformal map. This proves that there exists a neighbourhood of $\del M$ in $\hat{M}$ which is isometric to a tubular neighbourhood of the equator in $S^n$ equipped with a metric of the form $\mu^2 g_{S^n}$, where $\mu$ is $C^2$ and smooth away from the equator. Approximating $\mu$ with a suitable smooth function, we find a smooth metric $g^*$ on $\hat{M}$ conformal to $\hat{g}'$ with positive scalar curvature.\\

\emph{Step 2:} 
Let $\pi \colon N \to \hat{M}$ be the universal covering of $\hat{M}$. Equipped with the Riemannian metric induced by $g^*$, $N$ is a simply-connected complete locally conformally flat manifold with positive scalar curvature. From the results of \cite{SY} and \cite{SchoenYau}, it follows that the developing map $\phi \colon N \to S^n$ is injective and $\Lambda' \coloneqq S^n \setminus \phi(N)$ is the limit set of $\pi_1 (\hat{M})$. As $g^*$ has positive scalar curvature, the results of Schoen and Yau imply that the Hausdorff dimension of $\Lambda'$ is at most $\frac{n-2}{2}$. 

Let $\tilde{M} \subseteq N$ be a connected component of $\pi^{-1}(M)$. Then $\pi \colon \tilde{M} \to \nolinebreak M$ is a covering and since $\phi \colon N \to S^n$ was injective, $\Phi \coloneqq \phi|_{\tilde{M}}$ is also injective. Furthermore, the metrics $\tilde{g}$ and $\tilde{g}^*$ on $\tilde{M}$ induced by $g$ and $g^*$, resp., are conformally equivalent (since $g$ and $g^*$ are) and so $\Phi$ is also conformal with respect to $\tilde{g}$. We conclude that $\Phi \colon \interior(\tilde{M}) \to \Phi(\interior(\tilde{M}))$ is a conformal diffeomorphism. To show that it is actually a diffeomorphism of $\tilde{M}$, we need to verify that $\Phi$ is also a local diffeomorphism near the boundary, i.e. $\Phi(\del \tilde{M}) \subseteq \nolinebreak \del  \Phi(\tilde{M})$. 

To see this, note that  $\del \tilde{M}$ is diffeomorphic to disjoint copies of $S^{n-1}$, since the covering $\tilde{M} \to M$ induces a covering $\del \tilde{M} \to \del M \cong S^{n-1}$ and the latter is simply-connected. In particular, any connected component $S$ of $\del \tilde{M}$ is diffeomorphic to a sphere. Arguing as above, we see that for any such $S$,  $\Phi|_S \colon S \to \Phi(S)$ is a diffeomorphism, where $\Phi(S)$ is some geodesic sphere in $S^n$. Hence $S^n \setminus \Phi(S)$ has two connected components. As $\tilde{M} \setminus \del \tilde{M}$ is connected and $\Phi$ is injective, it follows that $\Phi(\tilde{M} \setminus \del \tilde{M})$ lies in exactly one of these connected components, thus $\Phi(\del \tilde{M}) \subseteq \del \Phi(\tilde{M})$\\

\emph{Step 3:} 
We have seen that $\Phi \colon \tilde{M} \to \Phi(\tilde{M})$ is a conformal diffeomorphism and $\del \Phi(\tilde{M}) = \Phi(\del \tilde{M}) \cup \Lambda$, where $\Lambda \subseteq \Lambda'$ is a closed subset. Note that $\Lambda$ is empty if $\tilde{M}$ is compact, i.e. $|\pi_1(M)| < \infty$. The set $\Phi(\tilde{M})$ can thus be identified with 
\[ 
C = C(\varepsilon_i, p_i,\Lambda) \coloneqq S^n \setminus \left( \bigcup_{i } D_{\varepsilon_i} (p_i) \cup \Lambda \right),  
\] 
where, by injectivity, the $D_{\varepsilon_i} (p_i)$ are geodesic balls in $S^n$ with disjoint closures and $\Lambda = \left( S^n \setminus  \bigcup_{i } D_{\varepsilon_i} (p_i) \right)  \cap \Lambda'$.  Then $\Phi \colon \tilde{M}  \to C \subseteq S^n$ is a conformal diffeomorphism as claimed. 
\end{proof}

\section{Extension of the metric} \label{Sec:Extension}

As we have seen, any $M$ satisfying the hypotheses of Theorem \ref{Thm:Theorema} can be conformally covered by a subset $C \subseteq S^n$ as above. We want to argue similarly to the simply-connected case, but certain problems arise: First of all, if $M$ is not simply-connected, then the image of $\tilde{M}$ under the developing map has more than one ``hole'' $D_{\varepsilon_i} (p_i)$ and if $\pi_1(M)$ is infinite, then the limit set $\Lambda$ is non-empty. In this section we show how to extend the metric $\tilde{h} \coloneqq (\Phi^{-1})^*\tilde{g}$ to $S^n \setminus \left( D_{\varepsilon_i} (p_i) \cup \Lambda \right)$ for a given particular $i$. The basic idea is to glue in spherical caps $\overline{D_{\pi - \rho}}$ along the boundaries; we make this construction more explicit below. We then show that the extended metric is isometric with the standard one contradicting the existence of more than one of those caps.\\

Pick any $i$.  For every $j \neq i$, we extend the metric $\tilde{h}$ to $D_{\varepsilon_j} (p_j)$ in the following way:  

Let $\{S, N=-S \}$ be any pair of antipodal points. Let $\phi_j$ be a conformal transformation of $S^n$ mapping $D_{\varepsilon_j} (p_j)$ to $D_{\pi - \rho}(N)$. By assumption, $(\Sigma_{\rho}(S),\phi_j^*\tilde{h}|_{\Sigma_{\rho}(S)})$ is isometric to $\Sigma_{\rho}(S)$ with the standard metric, therefore there exists an isometry $f_j \colon   \Sigma_{\rho}(S) \to (\Sigma_{\rho}(S),\phi_j^*\tilde{h}|_{\Sigma_{\rho}(S)})$. As $\phi_j^*\tilde{h}$ is conformal to the standard metric, $f_j$ is a Möbius transformation of $\Sigma_{\rho}(S)$. Using the Poincaré-extension (see Proposition \nolinebreak \ref{Prop:Poincare-Extension}), we can extend $f_j$ to a Möbius transformation of $S^n$ preserving $D_{\rho}(S)$, call this extension $F_j$. 

Then $h_j \coloneqq F_j^* \phi_j^* h$ is a metric on the set $C_j \coloneqq (\phi_j \circ F_j)^{-1}(C)$ which is of a similar form as $C$ (with different parameters), where one of the balls -- corresponding to the ball $D_{\varepsilon_j} (p_j)$ for our fixed $j$ -- is $D_{\pi - \rho}(N)$. As $h_j$ is conformal to the standard metric, we may write
\[
h_j = u_j^{\frac{4}{n-2}} g_{S^n},
\]
for some function $u_j$. By construction and assumption, we have that $h_j$ coincides with the standard metric on $\Sigma_{\rho}(S)$, $\Scal(h_j) \geq n(n-1)$ and the mean curvature of $\Sigma_{\rho}(S)$ is at least $H_\rho$. In terms of $u_j$, this becomes:
\[
\left\{\begin{array}{cl}
u_j = 1 &\text{on } \Sigma_{\rho}(S),\\
\frac{\del u_j}{\del \eta} \geq 0  &\text{on } \Sigma_{\rho}(S),\\
- \Delta u_j \geq  \frac{n(n-2)}{4} \left( u_j^{\frac{n+2}{n-2}} -  u_j \right)  &\text{in} \interior(C_j).
\end{array} \right.
\]
Hence the function $\bar{u}_j$ defined by
\[
\bar{u}_j(x) \coloneqq \left\{
\begin{array}{cl}
1 & \text{if } x \in \overline{D_{\pi - \rho}(N)},\\
u_j(x) & \text{if } x \in C_j
\end{array}\right.
\]
extends $u_j$ to $C_j \cup D_{\pi - \rho}(N)$ and is locally Lipschitz. Now extend $h_j$ by setting
\[
\bar{h}_j \coloneqq \bar{u}_j^{\frac{4}{n-2}} g_{S^n}.
\]
Pulling back $\bar{h}_j$ with $(\phi_j \circ F_j)^{-1}$, we obtain a metric $\tilde{h}_j$ on $C \cup D_{\varepsilon_j} (p_j)$ extending $h$. Combine all these by setting
\[
\tilde{h}(x) \coloneqq \tilde{h}_j(x) \text{ if } x \in C \cup D_{\varepsilon_j} (p_j).
\]

Using this extension, we show:

\begin{proposition} \label{Prop: Filling-the-holes}
The metric $h \coloneqq (\Phi^{-1})^*\tilde{g}$ can be extended to a metric $\tilde{h}$ on $S^n \setminus \left( D_{\varepsilon_i} (p_i) \cup \Lambda \right)$ which is locally Lipschitz and satisfies $\Scal(\tilde{h}) \geq n(n-1)$ weakly in the sense that $\tilde{h} = \tilde{u}^{\frac{4}{n-2}} g_{S^n}$ with
\[
-\Delta \tilde{u} \geq  \frac{n(n-2)}{4} \left( \tilde{u}^{\frac{n+2}{n-2}} -  \tilde{u} \right)
\]
weakly. With respect to $\tilde{h}$, the boundary $\del D_{\varepsilon_i} (p_i)$ has mean curvature at least $H_\rho$ and is isometric to $\Sigma_{\rho}$.
\end{proposition}

\begin{proof}
Writing $\tilde{h}$ constructed above as
\[
\tilde{h} = \tilde{u}^{\frac{4}{n-2}} g_{S^n},
\]
it remains to show that the extension $\tilde{u}$ satisfies
\begin{align} \label{Eq:Eq-for-u}
-\Delta \tilde{u} \geq  \frac{n(n-2)}{4} \left( \tilde{u}^{\frac{n+2}{n-2}} -  \tilde{u} \right)
\end{align} 
weakly. Note that, by construction, $\tilde{u}$ satisfies \eqref{Eq:Eq-for-u} in $\interior (C)$ and on every $D_{\varepsilon_j} (p_j)$, so we only need to take care of the boundary values. 

Let $H^j  =\cot(\varepsilon_j) >0 $ be the mean curvature of $D_{\varepsilon_j} (p_j)$ with respect to the spherical metric and computed with respect to the inner unit normal $\nu = -\eta$ (pointing into $D_{\varepsilon_j} (p_j)$). 

Since, by construction, $(D_{\varepsilon_j} (p_j), \tilde{h})$ is isometric to $D_{\pi - \rho}$ with the standard metric, the mean curvature of the boundary with respect to $\tilde{h}$ viewed from the inside is $-H_\rho$.
Writing
\[
 \frac{\del \tilde{u}}{\del \eta^\pm}(x) \coloneqq \lim_{\varepsilon \to 0^+}   \frac{\tilde{u}(\cos (\pm \varepsilon)x + \sin( \pm\varepsilon) \eta) - \tilde{u} (x)}{\pm \varepsilon}
\]
for the one-sided derivatives, this implies 
\[
 \frac{n-2}{2}(-H_\rho)\tilde{u}^{\frac{n}{n-2}}  = \frac{n-2}{2} H^j \tilde{u} + \frac{\del \tilde{u}}{\del \eta^-}
\]
on $\del D_{\varepsilon_j} (p_j)$. By hypothesis on $\tilde{h}$, we also have
\[
\frac{n-2}{2}H_\rho \tilde{u}^{\frac{n}{n-2}} \leq \frac{n-2}{2}H(\tilde{h}) \tilde{u}^{\frac{n}{n-2}} = - \frac{n-2}{2} H^j \tilde{u} + \frac{\del \tilde{u}}{\del \nu^-}.
\]
Adding these inequalities, we obtain
\[
\frac{\del \tilde{u}}{\del \eta^-} + \frac{\del \tilde{u}}{\del \nu^-} \geq 0, \text{  i.e.  } \frac{\del \tilde{u}}{\del \eta^-} \geq - \frac{\del \tilde{u}}{\del \nu^-} = \frac{\del \tilde{u}}{\del \eta^+}.
\] 
Let $\varphi \in \mathcal{D}^+(S^n \setminus ( \overline{D_{\varepsilon_i} (p_i)} \cup \Lambda ) ) = \{ \phi \in C^\infty_c (S^n \setminus ( \overline{D_{\varepsilon_i} (p_i)} \cup \Lambda )) \mid \phi \geq 0 \}$ be a smooth test function. Using Green's formula, we have:

\begin{align*}
\int_{S^n \setminus \left( D_{\varepsilon_i} (p_i) \cup \Lambda \right)} \hspace{-\vierz}&\hspace{\vierz}\tilde{u} (- \Delta \varphi)\\
 &=  \int_C  \tilde{u} (- \Delta \varphi) + \sum_{j\neq i} \int_{D_{\varepsilon_j} (p_j)} \tilde{u} (- \Delta \varphi) \\
&= \int_C  \varphi (- \Delta \tilde{u})   + \sum_{j\neq i} \int_{D_{\varepsilon_j} (p_j)} \varphi (- \Delta \tilde{u}) \\
&\qquad -  \int_{\del C} \left( \tilde{u}\frac{\del \varphi}{\del \nu}  - \varphi \frac{\del \tilde{u}}{\del \nu^-} \right) - \sum_{j\neq i} \int_{\del  D_{\varepsilon_j} (p_j)} \left( \tilde{u}\frac{\del \varphi}{\del \eta}  - \varphi \frac{\del \tilde{u}}{\del \eta^-} \right)\\
&=  \int_C  \varphi (- \Delta \tilde{u})   + \sum_{j\neq i} \int_{D_{\varepsilon_j} (p_j)} \varphi (- \Delta \tilde{u}) \\
& \qquad + \int_{\del C} \varphi \left( \frac{\del \tilde{u}}{\del \eta^-}- \frac{\del \tilde{u}}{\del \eta^+}  \right)\\
&\geq \frac{n(n-2)}{4} \left[ \int_C  \varphi   \left( \tilde{u}^{\frac{n+2}{n-2}} -  \tilde{u} \right) +  \sum_{j\neq i} \int_{D_{\varepsilon_j} (p_j)} \varphi \left( \tilde{u}^{\frac{n+2}{n-2}} -  \tilde{u} \right) \right]\\
&= \frac{n(n-2)}{4} \int_{S^n \setminus \left( D_{\varepsilon_i} (p_i) \cup \Lambda \right)}  \left( \tilde{u}^{\frac{n+2}{n-2}} -  \tilde{u} \right) \varphi.
\end{align*}
\end{proof}

Using this extension, we conclude the argument:

\begin{proposition} \label{Prop:M-is-simply-connected}
Under the assumptions of Theorem \ref{Thm:Theorema}, $M$ is simply-connected and isometric to $\overline{D_\rho}$. 
\end{proposition}

\begin{proof} Let $\tilde{h}$ be the metric constructed in Proposition \ref{Prop: Filling-the-holes}.
As above (i.e.\ by pulling back $\tilde{h}$ with a Möbius transformation), we may assume that $D_{\varepsilon_i} (p_i) = D_{\pi- \rho}(N)$ and $\tilde{h}$ restricted to the boundary $\del D_{\pi- \rho}(N) = \Sigma_{\rho} (S)$ coincides with the restriction of the standard metric $g_{S^n}$. 

As in the simply-connected case, it is convenient to transfer the problem to $\R^n$ via stereographic projection. Let $\pi \colon S^n \setminus \{ N \} \to  \R^n$ be the stereographic projection from $N$, then $\pi(D_\rho(S)) = B_{r}(0) \eqqcolon B_r$, where $r = \tan (\rho/2)$.

Note that $\pi^*g_{S^n} = w^{\frac{4}{n-2}} \sum \d x^i \otimes \d x^i$, where
\[
w(x) \coloneqq \left(\frac{2}{1+ |x|^2} \right)^\frac{n-2}{2}.
\]
We consider the metric $\pi^* \tilde{h}$, which is a metric on the set $\overline{B_{r}} \setminus \pi(\Lambda)$. Writing
\[
\pi^* \tilde{h} =  v^{\frac{4}{n-2}} \sum_{i = 1}^n \d x^i \otimes \d x^i,
\]
we obtain $v \in C^{0, 1}_{\text{loc}}( \overline{B_r} \setminus \pi(\Lambda))$ positive satisfying
\begin{equation} \label{Eq:Eq-for-v}
\left\{\begin{array}{cl}
v = \left( \frac{2}{1+r^2} \right)^{\frac{n-2}{2}} &\text{on } \del B_r,\\
- \Delta v \geq  \frac{n(n-2)}{4} v^{\frac{n+2}{n-2}}  &\text{weakly in } B_r \setminus \pi(\Lambda).
\end{array} \right.
\end{equation}

\begin{remark}
The latter equation actually holds on the whole of $B_r$, see \cite[Theorem 5.1]{SY} 
\end{remark}

Let $m \geq 2^{\frac{n-2}{2}}$ and $\Lambda_m \subseteq \left( v^{-1}((m, \infty)) \cup \pi( \Lambda) \right)$ be an open set with smooth boundary which contains $\pi(\Lambda)$. Such $\Lambda_m$ exists as $v(x) \to \infty$ for $x \to \pi(\Lambda)$, see \cite[Proposition 2.6]{SY}. If $\Lambda = \emptyset$, we set $\Lambda_m \coloneqq \emptyset$. Set $\Omega_m \coloneqq \overline{B_r} \setminus \Lambda_m$.
 
Since $- \Delta v \geq \frac{n(n-2)}{4} v^{\frac{n+2}{n-2}} > 0 $, $v$ is superharmonic. Hence the maximum principle implies
\[
v \geq \inf_{x \in \del \Omega_m} v(x)  = \left(\frac{2}{1+ r^2} \right)^\frac{n-2}{2}.
\]
Consider the boundary value problem

\begin{equation} \label{Eq:Eq-for-v2}
\left\{\begin{array}{cl}
- \Delta f = \frac{n(n-2)}{4} f^{\frac{n+2}{n-2}} &\text{ in } \Omega_m ,\\
f(x) = \left(\frac{2}{1+ |x|^2} \right)^\frac{n-2}{2}   &\text{on } \del \Omega_m.\\
\end{array} \right.
\end{equation}
As $m \geq 2^{\frac{n-2}{2}}$,  $v$ is a supersolution while the constant function $\left(\frac{2}{1+ r^2} \right)^\frac{n-2}{2}$ is a subsolution. Hence there exists a solution $v'$ of  \eqref{Eq:Eq-for-v2} with 
\[ 
v \geq v' \geq \left(\frac{2}{1+ r^2} \right)^\frac{n-2}{2} .
\]
 (see e.g. \cite{ClementSweers}). By standard regularity theory, $v' \in C^2(\overline{\Omega_m})$.

Consider the metric $v'^{\frac{4}{n-2}} \sum \d x^i \otimes \d x^i$. Identifying the standard spherical metric with $w^{\frac{4}{n-2}} \sum \d x^i \otimes \d x^i$, we can apply \cite[Proposition 1]{HangWang2} to conclude
\[
v'(x) \geq w(x) = \left(\frac{2}{1+ |x|^2} \right)^\frac{n-2}{2}
\]
and hence $v \geq w$.
Since 
\[
- \Delta(v - w) \geq \frac{n(n-2)}{4} \left( v^{\frac{n+2}{n-2}} - w^{\frac{n+2}{n-2}} \right) \geq 0,
\]
$v- w$ is superharmonic. As $v -w  \geq 0$ and $v = w$ on $\del B_r$, the Hopf lemma implies $\frac{\del}{\del \eta} (v - w) \leq 0$ there. Let $H^r = r^{-1}$ be the mean curvature of $\del B_r$ with respect to the euclidean metric. Then on $\del B_r$, we have
\begin{align*}
\frac{\del v}{\del \eta} &=   \frac{n-2}{2} H(g) v^{\frac{n}{n-2}}  - \frac{n-2}{2} H^r v \\
&\geq \frac{n-2}{2} H_\rho  v^{\frac{n}{n-2}}  - \frac{n-2}{2} H^r v\\
&= \frac{n-2}{2} H_\rho w^{\frac{n}{n-2}}  - \frac{n-2}{2} H^r w = \frac{\del w}{\del \eta},
\end{align*}
which implies $v= w$ on the connected component of $\del B_r$ by the Hopf lemma. For large $m$, this component must contain some of the $\pi(D_{\varepsilon_j} (p_j))$, if there were any.  This would contradict the fact that, with respect to $\pi^*h$, those are isometric to $\Sigma_\rho$. Hence there are no $D_{\varepsilon_j} (p_j)$, which implies that $\del \tilde{M}$ is connected and therefore $M$ is simply-connected. Thus also $\Lambda = \emptyset$ and $v = w$ on the whole of $B_r$.
\end{proof}

\bibliography{bibliography}{}
\bibliographystyle{amsalpha}

\end{document}